\documentclass[reprint,
onecolumn,
superscriptaddress,
showpacs,preprintnumbers,
showkeys
nobibnotes,
amsmath,amssymb,
showkeys,
floatfix,
]{revtex4-1}
\usepackage{graphicx}
\usepackage{epstopdf}
\usepackage{graphicx,graphics,color,epsfig}
\usepackage{revsymb4-1}
\usepackage{bm}
\usepackage{braket}
\usepackage{mathptmx}
\usepackage[colorlinks, linkcolor=blue, anchorcolor=blue, citecolor=blue]{hyperref}
\usepackage{hyperref}
\usepackage{amsmath}%
\usepackage{amsfonts}%
\usepackage{amssymb}%

\newtheorem{theorem}{Theorem}

\newtheorem{lemma}[theorem]{Lemma}

\newenvironment{proof}[1][Proof]{\textbf{#1.} }{\ \rule{0.5em}{0.5em}}
\numberwithin{equation}{section}

\newcommand{\orcid}[1]{\href{https://orcid.org/#1}{\textcolor[HTML]{A6CE39}{\aiOrcid}}}
\DeclareMathOperator {\sech}{sech}

\begin{document}

\title{Solutions of the two-wave interactions in quadratic nonlinear media}


\author{Lazhar Bougoffa }
\email{lbbougoffa@imamu.edu.sa}
\affiliation{Imam Mohammad Ibn Saud Islamic
University (IMSIU), College of  Science, Department of Mathematics,
P.O. Box 90950, Riyadh 11623, Saudi Arabia\\
ORCiD:  http://orcid.org/0000-0003-0189-9436}

\author{Smail Bougouffa }
\email{sbougouffa@imamu.edu.sa}
\affiliation{Imam Mohammad Ibn Saud Islamic
Imam Mohammad
Ibn Saud Islamic University (IMSIU), College of  Science, Department
of Physics, P.O. Box 90950, Riyadh 11623, Saudi Arabia\\
ORCiD:  http://orcid.org/0000-0003-1884-4861}

\date{\today}
\begin{abstract}
In this paper, we propose a reliable treatment for studying the
two-wave (symbiotic) solitons of interactions in nonlinear quadratic
media. We investigate the Schauder's fixed point theorem for proving
the existence theorem. Additionally, the uniqueness solution for
this system is proved. Also, a highly accurate approximate solution
is  presented via an iteration algorithm.
\end{abstract}

\keywords{Two-wave solitons; existence and uniqueness
solutions; exact solution; approximate solution}  

\maketitle

\setcounter{section}{-0} 

\section{Introduction}\label{Int}
In the simplest case of type-I second harmonic generation (SHG)
without walk-off between harmonic waves  soliton evolution is
described by the normalized system \cite{1,2,3,4,5,6,7}:
\begin{eqnarray}
\left\{
\begin{array}{rcl}
i\frac{\partial\varphi}{\partial z}+r\frac{\partial^{2}\varphi}{\partial x^{2}}-\varphi+\varphi^{*}\psi &=&0,\\
i\sigma\frac{\partial\psi}{\partial z}+
s\frac{\partial^{2}\psi}{\partial
x^{2}}-\alpha\psi+\frac{1}{2}\varphi^{2} &=&0,
\end{array}
\right.
\end{eqnarray}
where $\alpha$ is the rescaled soliton parameter and satisfies
$\alpha=\sigma(2\beta+\lambda)/\beta ,$ the dimensionless parameter
$\beta$ is the normalized nonlinearity-induced shift to the
propagation constant of the fundamental harmonic wave, $\sigma$ and
$\lambda$ are the coefficient and  phase mismatch parameter,
respectively. This system represents the generic model of $\chi^{2}$
solitons. There are other types of normalization also used in the
literature see e.g. \cite{8,9,10}. The solutions of this system have
been discussed for one dimensional in \cite{3,4,5, 11, 12,13,14} and
multi-dimensional cases in \cite{15}. Under suitable assumptions,
the problem of the two-wave (symbiotic) solitons  can be reduced to
the solution of the following coupled system \cite{1,2,3,4,7, 14,
16}
\begin{eqnarray}
\left\{
\begin{array}{rcl}
r\frac{d^{2}\varphi}{dx^{2}}-\varphi+\varphi\psi &=&0,\\
s\frac{d^{2}\psi}{dx^{2}}-\alpha\psi+\frac{1}{2}\varphi^{2} &=&0,
\end{array}
\right.\label{1}
\end{eqnarray}
The properties of solitons described by system {(\ref{1})} are well
known see \cite {3,4,5}, where the authors Buryak and Kivshar \cite
{3,4,5} looked for stationary (i.e. z-independent) localized
solutions of the normalized system in the form of an asymptotic
series in the parameter $\alpha^{-1}$ and found the real functions
$\varphi(x)$ and $\psi(x)$ in the form of asymptotic series:
\begin{equation}
\phi(x)= 2\alpha^{\frac{1}{2}}\sech ( x)
+4s\alpha^{-\frac{1}{2}}\tanh^{2}(x) \sech ( x)+...,
\end{equation}
\begin{equation}
\psi(x)= 2\sech^{2} (x) +s\alpha^{-1}\left(16 \sech^{2}( x) -20\sech^{4}
(x)\right)+...,
\end{equation}
for bright solitons at $r = +1,$ and
\begin{equation}
\phi(\tau)=\sqrt{2}\alpha^{\frac{1}{2}}\tanh (\tau)
\sqrt{2}s\alpha^{-\frac{1}{2}}\left(\tau \sech^{2}( \tau )-\tanh (\tau)
\sech^{2}( \tau)\right)+...,
\end{equation}
\begin{equation}
\psi(x)= \tanh^{2}(\tau) +s\alpha^{-\frac{1}{2}}\left(2\tau \tanh
(\tau)\sech^{2}( \tau) -4 \sech^{2}( \tau) +5\sech^{4}( \tau)\right)+...,
\end{equation}
where $\tau=\frac{x}{\sqrt{2}},$ for dark solitons at $r = -1.$\\
Exact solutions of system {(\ref{1})} have been found at $\alpha=1$
for $r=s=1,$
$r=s=-1$ and $r=-s=-1$ in \cite {11, 14}. Another  solution of
{(\ref{1})} in the case $r=1,$ $s=-1$ and $\alpha=2$ is provided in
an explicit analytical form \cite {12}.
Also, different analytical approximation methods have been proposed
to deal with the system {(\ref{1})}. For example,
an accurate approximate  solution 
with the help of the variational method is obtained in \cite {14}
and a family of bright (dark) solitons for $s=r=1,\ s=-r=1$ and
$\alpha>0$ was also discussed  in \cite {16} using the numerical
shooting and relaxation techniques. \\ Recently, the authors
\cite{17} studied this coupled system subject to the following
boundary conditions:
\begin{equation}
\label{2}\phi(l_{1})=\phi(l_{2}) =0 \ \mbox{and}\
\psi(l_{1})=\psi(l_{2}) =0
\end{equation}
and showed that the solutions of this system could be obtained
numerically using the Green's function method.\\
Since the exact analytical solutions of Eqs.{(\ref{1})} cannot be
found for arbitrary values of $\alpha,$ then the purpose of this
work is to present a result of the existence and uniqueness of
solutions. Also, an exact implicit solution is derived using a
useful procedure at $\alpha=1.$ Thus, the paper is organized as
follows: in Section \ref{Sec2}, we investigate the existence and
uniqueness theorem of the two-wave solitons in quadratic media,
where the problem is formulated in the context of two nonlinear
coupled differential equations in one dimension. Then, in Section
\ref{Sec3}, we solve the coupled system by an appropriate technique
with suitable boundary conditions. A systematic numerical procedure
is proposed in Section \ref{Sec4}. Finally, we conclude with some
remarks in Section \ref{Sec5}.

\section{An existence and uniqueness theorem}\label{Sec2}
Rewrite (\ref{1}) in the following system
\begin{eqnarray}
\left\{
\begin{array}{rcl}
\frac{d^{2}\varphi}{dx^{2}} &=&f_{1}(\varphi,\psi),\\
\frac{d^{2}\psi}{dx^{2}} &=&f_{2}(\varphi,\psi),
\end{array}
\right.\label{3}
\end{eqnarray}
where $f_{i}:[l_{1},l_{2}]\times \mathbb{R}\times
\mathbb{R}\longrightarrow\mathbb{R}, i=1,2$ are defined by
$f_{1}(\varphi,\psi)=\frac{1}{r}(\varphi-\varphi\psi)$ and
$f_{2}(\varphi,\psi)=\frac{1}{s}(\alpha\psi-\frac{1}{2}\varphi^{2}).$
\\
\textbf{Existence -}In this section, we shall deal with the
existence of solutions of the BVP {(\ref{1})} with {(\ref{2})}. First off
all, we shall prove  the following lemmas, which  are useful tools
in the proof of the existence and uniqueness theorem.
\begin{lemma}
If we assume that $\varphi, \ \psi$ $\in \mathbb{C}[l_{1},l_{2}], \
l_{2}>l_{1}.$ Then, the functions $f_{i}(\varphi,\psi), \ i=1,2$ are
Lipschitz continuous functions of $\varphi$ and $\psi.$
\end{lemma}
\begin{proof}
From the definition of $f_{i}(\varphi,\psi),$ we have
\begin{eqnarray}\mid
f_{1}(\varphi_{1},\psi_{1})-f_{1}(\varphi_{2},\psi_{2})\mid=\frac{1}{\mid
r \mid}\mid
(\varphi_{1}-\phi_{2})+\psi_{2}(\varphi_{2}-\varphi_{1})+\varphi_{1}(\psi_{2}-\psi_{1})\mid.
\end{eqnarray}
Since $\varphi_{i}, \ \psi_{i}$ $\in \mathbb{C}[l_{1},l_{2}],$ then
there exist  $ M_{i}, M^{*}_{i} > 0, \ i=1,2$  such that $\mid
\varphi_{i}(x)\mid\leq M_{i}$ and $\mid \psi_{i}(x)\mid\leq
M^{*}_{i}, \ M_{i}, M^{*}_{i} > 0, \ i=1,2$ for all $x\in [l_{1},
l_{2}].$
\begin{eqnarray}\mid
f_{1}(\varphi_{1},\psi_{1})-f_{1}(\varphi_{2},\psi_{2})\mid\leq
\frac{1}{\mid r \mid}\left[\mid
\varphi_{2}-\varphi_{1}\mid+M^{*}_{2}\mid\varphi_{2}-\varphi_{1}\mid+M_{1}\mid
\psi_{2}-\psi_{1}\mid\right].
\end{eqnarray}
Hence,
\begin{eqnarray}\mid
f_{1}(\varphi_{1},\psi_{1})-f_{1}(\varphi_{2},\psi_{2})\mid\leq
L_{1,1}\mid \varphi_{2}-\varphi_{1}\mid+L_{1,2}\mid
\psi_{2}-\psi_{1}\mid.
\end{eqnarray}
Similarly, we obtain
\begin{eqnarray}\mid
f_{2}(\varphi_{1},\psi_{1})-f_{2}(\varphi_{2},\psi_{2})\mid\leq
L_{2,1}\mid \varphi_{2}-\varphi_{1}\mid+L_{2,2}\mid
\psi_{2}-\psi_{1}\mid,
\end{eqnarray}
where $L_{1,1}=\frac{1}{\mid r \mid}(1+M^{*}_{2}),$
$L_{1,2}=\frac{M_{1}}{\mid r \mid},$ $L_{2,1}=\frac{
M_{1}+M_{2}}{2\mid s \mid}$ and $L_{2,2}=\frac{\alpha}{\mid s
\mid}.$
\end{proof}
\begin{lemma} ( See pp. 70-71 \cite{18})\\
Let $g:[a,b]\rightarrow \mathbb{R}$ be a continuous function. The
unique solution $u$ of the following  boundary value problem
\begin{equation}
u''=g(x)
\end{equation}
subject to the Dirichlet boundary conditions  $u(a)=u(b)=0$ is given
by
\begin{equation}
u(x)=\int_{a}^{b}G(x,y)g(y)dy,
\end{equation}
where $G(x,y)$ is the Green function given by
\begin{eqnarray}
G(x,y)=\left\{
\begin{array}{rcl}
\frac{1}{b-a}(x-a)(b-y), \  \ a \leq x \leq y \leq b, \\
\frac{1}{b-a}(y-a)(b-x), \   a \leq y \leq x \leq b
\end{array}
\right.
\end{eqnarray}
and   $\int_{a}^{b}\mid G(x,y)\mid dy\leq \frac{(b-a)^{2}}{8}.$
\end{lemma}
 Replacing $g(x),$ $a$ and $b$ by
$f_{i}(\varphi,\psi),$ $l_{1}$ and $l_{2}$   in Lemma 2,
respectively, we obtain an equivalent integral system
\begin{eqnarray}
 \left\{
\begin{array}{rcl}
\varphi(x)&=&\int_{l_{1}}^{l_{2}}G(x,s)f_{1}(\varphi(s),\psi(s))ds,\ x\in [l_{1},l_{2}],\\
\psi(x)&=&\int_{l_{1}}^{l_{2}}G(x,s)f_{2}(\varphi(s),\psi(s))ds,\
x\in [l_{1},l_{2}].
\end{array}
\right.\label{5}
\end{eqnarray}
 Define the Banach space
$\mathbb{X}=\mathbb{C}[l_{1},l_{2}]$ with norm $\|(\varphi, \psi)\|=
\|\varphi\|+\|\psi\|,$ where $\|u\|=\max_{l_{1}\leq x\leq l_{2}}\mid
u(x)\mid$ and the operator $T: \mathbb{X} \longrightarrow
\mathbb{X}$  by $T(\varphi, \psi)=\left(T_{1}(\varphi, \psi),
T_{2}(\varphi, \psi)\right),$ where
  \begin{equation}
    T_{1}(\varphi, \psi) = \int_{l_{1}}^{l_{2}}G(x,s)f_{1}(\varphi(s),\psi(s))ds
  \end{equation}
  and
 \begin{equation}
    T_{2}(\varphi, \psi) = \int_{l_{1}}^{l_{2}}G(x,s)f_{2}(\varphi(s),\psi(s))ds.
  \end{equation}
  Since  $\mid \varphi(x)\mid\leq M$ and $\mid \psi(x)\mid\leq M^{*}, \ x\in [l_{1}, l_{2}],$ consider the closed and convex set
  \begin{eqnarray}
    \mathbb{S}=\left\{(\varphi, \psi)\in \mathbb{X}:\  \|(\varphi, \psi)\|\leq M+M^{*}
    \right\}.
  \end{eqnarray}
Furthermore, assume that $l_{2}-l_{1}\leq\sqrt[3]{\frac{8\mid
rs\mid(M+M^{*})}{\mid s\mid(M+M M^{*})+\mid r\mid(\alpha
M^{*}+\frac{M^{2}}{2})}}.$\\
In the theory of differential equations, there are a lot of
methods to establish the existence of solutions. Theorems concerning
the existence and properties of fixed points  are known as
fixed-point theorems. Such theorems are the most important tools for
proving the existence and uniqueness of the solution.  The
fundamental theorem used in this theory is Schauder' s theorem. In
order to make use of this
 theorem, it is sufficient to prove the
following lemma.
\begin{lemma}
For any $(\varphi, \psi)\in \mathbb{S},$ $T(\varphi, \psi)$ is
contained in $\mathbb{S}.$
\end{lemma}
\begin{proof}
It follows by the definition of $T(\varphi, \psi)$ that
\begin{equation}
    \mid T_{1}(\varphi(x), \psi(x))\mid \leq \int_{l_{1}}^{l_{2}}\mid G(x,s)\mid \mid f_{1}(\varphi(s),\psi(s))\mid
    ds\leq \frac{(l_{2}-l_{1})^{3}}{8\mid r\mid}(M+MM^{*}),
  \end{equation}
\begin{equation}
    \mid T_{2}(\varphi(x), \psi(x))\mid \leq \int_{l_{1}}^{l_{2}}\mid G(x,s)\mid \mid f_{2}(\varphi(s),\psi(s))\mid
    ds \leq \frac{(l_{2}-l_{1})^{3}}{8\mid s \mid}(\alpha
    M^{*}+\frac{M^{2}}{2}).
  \end{equation}
  Thus \begin{equation}
    \mid T(\varphi(x), \psi(x))\mid \leq \frac{(l_{2}-l_{1})^{3}}{8}\left[\frac{M+MM^{*}}{\mid r\mid}+\frac{\alpha
    M^{*}+\frac{M^{2}}{2}}{\mid s\mid}\right].
  \end{equation}
  Since $l_{2}-l_{1}$ is defined by  the above condition, thus $\mid T(\varphi(x), \psi(x))\mid \leq
  M+M^{*}.$  On account of the continuity  of $f_{i}(\varphi(x), \psi(x)),$ $\varphi$ and $\psi,$ it follows that
   $T(\varphi(x), \psi(x))$ is continuous. This shows that $T(\varphi(x),
   \psi(x))$ is also contained in $\mathbb{S}.$
\end{proof}

In order to prove that $T(\varphi(x), \psi(x))$ is equicontinuous, it is easy to see from its definition that
\begin{equation}
    \mid T(\varphi(x), \psi(x))-T(\varphi(x'), \psi(x'))\mid \leq K\mid
    x-x'\mid, \ \mbox{for \ any } x, x'\in [l_{1},l_{2}].
  \end{equation}
  where $K=\frac{(l_{2}-l_{1})^{2}}{8}\left[\frac{M+MM^{*}}{\mid r\mid} +\frac{\alpha
    M^{*}+\frac{M^{2}}{2}}{\mid s\mid}\right].$\\
    Therefore $T$ is compact by the classical Ascoli lemma, and
   Schauder's fixed point theorem yield the fixed point of $T.$
   Thus, we have proved:
    \begin{theorem}
    There exists a continuous solution $(\varphi, \psi)$ which
    satisfies system  Eq. (\ref{1})  and   Eq.(\ref{2}) with the condition on
    $l_{1}$ and $l_{2}.$
  \end{theorem}
\textbf{Uniqueness -} A uniqueness theorem can also be obtained from
the Lipschitz continuous  in $\varphi$ and $\psi.$
\begin{theorem}
  If $\max\left(\max(L_{1,1},L_{2,1}),\max(L_{1,2},L_{2,2})
  \right)<\frac{8}{(l_{2}-l_{1})^{2}},$ then the  system Eq. (\ref{1}) with Eq. (\ref{2})  has a unique solution $(\varphi(x), \psi(x)).$
  \end{theorem}
\begin{proof}
Let $(\varphi_{1}, \psi_{1})$ and $(\varphi_{2}, \psi_{2})$ be two
solutions of (\ref{1})-(\ref{2}). Then, for $x\in [l_{1},l_{2}],$
\begin{equation}
    \mid \varphi_{2}-\varphi_{1}\mid \leq \frac{(l_{2}-l_{1})^{2}}{8}\left[
L_{1,1}\max_{l_{1}\leq x\leq l_{2}}\mid
\varphi_{2}(x)-\varphi_{1}(x)\mid +L_{1,2}\max_{l_{1}\leq x\leq
l_{2}}\mid \psi_{2}(x)-\psi_{1}(x)\mid \right],
  \end{equation}
\begin{equation}
    \mid \psi_{2}-\psi_{1}\mid \leq \frac{(l_{2}-l_{1})^{2}}{8}\left[
L_{2,1}\max_{l_{1}\leq x\leq l_{2}}\mid
\varphi_{2}(x)-\varphi_{1}(x)\mid +L_{2,2}\max_{l_{1}\leq x\leq
l_{2}}\mid \psi_{2}(x)-\psi_{1}(x)\mid \right].
  \end{equation}
 Consequently,
 \begin{equation}
    \| \varphi_{2}-\varphi_{1}\| +\| \psi_{2}-\psi_{1}\| \leq
    A\left[
\| \varphi_{2}-\varphi_{1}\|+\| \psi_{2}-\psi_{1}\|\right],
  \end{equation}
  where $A= \frac{(l_{2}-l_{1})^{2}}{8}\max\left(\max(L_{1,1},L_{2,1}),\max(L_{1,2},L_{2,2})
  \right).$
  We now apply the condition $A<1$ to this inequality, we get $\| \varphi_{2}-\varphi_{1}\|
  =0$ and $\| \psi_{2}-\psi_{1}\| =0.$ Therefore $(\varphi_{1}, \psi_{1})=(\varphi_{2}, \psi_{2}).$
  \end{proof}

The proof is complete.

\section{On the decoupling  of  the system {(\ref{1})}}\label{Sec3}

In this section, first of all, we are concerned with the norms
estimate for the functions $\varphi$ and $\psi$ when  $r=s=1.$
\begin{lemma}
Let $\varphi$ and $\psi$ be two functions in $\mathbb{L}_{2}(I),$
where $I=[l_{1},l_{2}].$ Then
\begin{eqnarray}
 \left\{
\begin{array}{rcl}
\|\varphi\|_{1}< \sqrt{2}\|\psi\|_{1}, \ \mbox{ if } \ \alpha <1,\\
\|\varphi\|_{1}= \sqrt{2}\|\psi\|_{1}, \ \mbox{ if } \ \alpha =1,\\
\|\varphi\|_{1}> \sqrt{2}\|\psi\|_{1}, \ \mbox{ if } \ \alpha >1,
\end{array}
\right.\label{15}
\end{eqnarray}
where $\|.\|_{1}$ is the norm defined in the Sobolev space
$\mathbb{H}^{1}_{0}(I)$ by
\begin{eqnarray}
\|u\|^{2}_{1}=\int_{l_{1}}^{l_{2}}\left(u^{2}(x)+ u'^{2}(x)\right)dx, \quad
\mbox{with}\ u(l_{1})=u(l_{2})=0.
\end{eqnarray}
\end{lemma}
\begin{proof}
Multiplying both sides of the first equation of system {(\ref{1})}
by $\varphi$ and integrating from $l_{1}$ to  $l_{2},$ we obtain
\begin{equation}
\label{16}\int_{l_{1}}^{l_{2}}\varphi''(x)\varphi(x)dx-\int_{l_{1}}^{l_{2}}\varphi^{2}(x)dx+\int_{l_{1}}^{l_{2}}\psi(x)\varphi^{2}(x)=0.
\end{equation}
From the second equation of system {(\ref{1})}, we have
\begin{equation}
\label{17}\varphi^{2}= -2\frac{d^{2}\psi}{dx^{2}}+2\alpha\psi.
\end{equation}
By substitution into the last term of {(\ref{16})}, we obtain
\begin{equation}
\label{18}\int_{l_{1}}^{l_{2}}\varphi''(x)\varphi(x)dx-\int_{l_{1}}^{l_{2}}\varphi^{2}(x)dx-2\int_{l_{1}}^{l_{2}}\psi''(x)\psi(x)dx+2\alpha\int_{l_{1}}^{l_{2}}\psi^{2}(x)dx=0.
\end{equation}
Integrating by parts and taking into account the given boundary
conditions, we obtain
\begin{equation}
\label{19}\int_{l_{1}}^{l_{2}}\varphi^{2}(x)dx+\int_{l_{1}}^{l_{2}}\varphi'^{2}(x)dx=
2\alpha\int_{l_{1}}^{l_{2}}\psi^{2}(x)dx+2\int_{l_{1}}^{l_{2}}\psi'^{2}(x)dx.
\end{equation}
This gives {(\ref{15})}.
\end{proof}

Let us now consider the case  $\alpha=1$ \cite{7,11, 14} and in view
of Lemma 6 if $\varphi(x)=\sqrt{2}\psi(x),$ then it may be shown
that the two equations of system {(\ref{1})} can be separated into
the following nonlinear equation
 \begin{equation}
\label{20}\psi''(x)-\psi(x)+\psi^{2}(x)=0.
\end{equation}
The exact solution to Eq.{(\ref{20})} follows by simply multiplying
both sides of Eq.{(\ref{20})} by $\psi'.$
\begin{eqnarray}
\label{21}\psi'\psi''-\psi \psi'+\psi^{2}=0,
\end{eqnarray}
which can be written as follows
\begin{eqnarray}
\frac{1}{2}\frac{d}{dx}\left(\psi'\right)^{2}-\frac{1}{2}\frac{d}{dx}\left(\psi\right)^{2}+\frac{1}{3}\frac{d}{dx}\left(\psi\right)^{3}=0.
\end{eqnarray}
and integrating with respect to $x,$ we obtain
\begin{eqnarray}
\label{23}
\left(\psi'\right)^{2}-\psi^{2}+\frac{2}{3}\psi^{3}=c_{1},
\end{eqnarray}
where $c_{1}$ is an arbitrary constant of integration. Thus
\begin{eqnarray}
\label{24} \psi'= \pm\sqrt{\psi^{2}-\frac{2}{3}\psi^{3}+c_{1}}.
\end{eqnarray}
 In view of
$\psi'(x)=\frac{d\psi}{d x},$ we have
\begin{equation}
\label{25}\pm\frac{d\psi}{\sqrt{\psi^{2}-\frac{2}{3}\psi^{3}+c_{1}}}=
 dx.
\end{equation}
Consequently,
\begin{equation}
\label{26}\pm\int
\frac{d\psi}{\sqrt{\psi^{2}-\frac{2}{3}\psi^{3}+c_{1}}}=  x+ c_{2},
\end{equation}
where $c_{2}$ is also a  constant of integration.\\
The LHS of Eq.{(\ref{26})} can be evaluated direct from the
integrals of irrational functions. Indeed, if we choose $c_{1}=0,$
then
\begin{eqnarray}
 \int \frac{d\psi}{\sqrt{\psi^{2}-\frac{2}{3}\psi^{3}}}=\mp2\tanh^{-1}\left(\sqrt{1-\frac{2}{3}\psi}\right)= x+ c_{2}.
\end{eqnarray}
Since $\tanh^{2}(\mp x)=\tanh^{2}(x). $ Hence,  a simple computation
leads to the implicit solution
\begin{eqnarray}
\label{27}(\varphi,\psi)=\left(\frac{3}{2}\left(1-\tanh^{2}(\frac{1}{2}
(x+c_{2}))\right), \frac{3}{\sqrt{2}}\left(1-\tanh^{2}(\frac{1}{2}
(x+c_{2}))\right)\right).
\end{eqnarray}
Thus, we have
\begin{lemma}
The system (\ref{1}) can be decoupled without increase the order of
the system  into the nonlinear equation (\ref{20}) when $\alpha=1.$
Furthermore, the solution $(\varphi, \ \psi)$ is given by
(\ref{27}).
\end{lemma}

In figure \ref{Fig1}, we display the variation of the exact solutions Eqs. (\ref{27}) in terms of the independent variables $x$ for different values of the constant of integration $c_2$ (Eq. (\ref{26} )). it can be seen that this constant of integration shifted left or right the distribution away from the origin with negative or positive values of $c_2,$ respectively. Also, it does not affect the behavior of the solutions, and the maximum value of the solution remains unchanged. Thus it can be chosen $c_2=0.$

\begin{figure}[h]
\hspace*{1.cm}\textbf{(a)} \hspace*{4.cm}\textbf{(b)}  \hspace*{4.cm}\textbf{(c)} \\
\includegraphics[width=0.33\linewidth,height=0.35\linewidth]{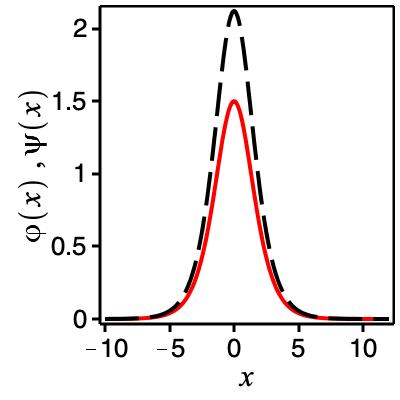}~ \includegraphics[width=0.33\linewidth,height=0.35\linewidth]{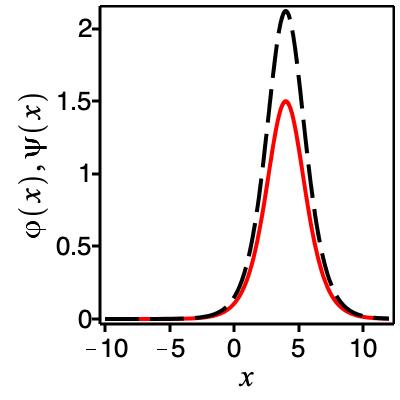}~ \includegraphics[width=0.33\linewidth,height=0.35\linewidth]{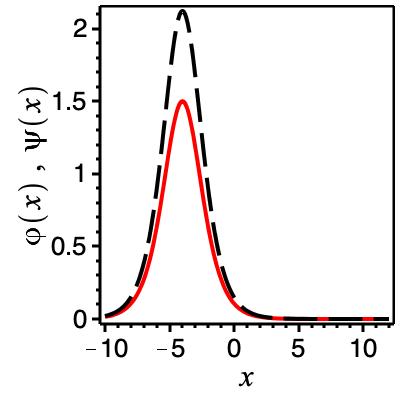}
 \caption{(Color online) The exact solutions (Eqs. \ref{27}) for different values of the integration constant $c_2$. (a) for $c_2=0$. (b) for $c_2=2$ and (c) for $c_2=-2$. For other parameters (see text) $r=s=1$ and $\alpha=1$. Solid red line $\varphi(x)$, Dashed black line $\psi(x)$.}\label{Fig1}
\end{figure}

\section{Numerical analysis}\label{Sec4}
An integral system equivalent to Eq. {(\ref{1})} with Eq. {(\ref{2})} can be
derived. Indeed, integrating (\ref{1}) twice from $l_{1}$ to $x$ and
taking into account the boundary conditions
$\varphi(l_{1})=\psi(l_{1})=0,$ we obtain
\begin{eqnarray}
 \left\{
\begin{array}{rcl}
\varphi&=&\beta(x-l_{1})+\int_{l_{1}}^{x}\int_{l_{1}}^{y}f_{1}(\varphi(s),\psi(s))dsdy,\\
\psi&=&\gamma(x-l_{1})+\int_{l_{1}}^{x}\int_{l_{1}}^{y}f_{2}(\varphi(s),\psi(s))dsdy,
\end{array}
\right.
\end{eqnarray}
where  $\beta=\varphi'(l_{1})$ and $\gamma=\psi'(l_{1})$ are unknown
constants to be determined from the second boundary conditions
$\varphi(l_{2})=\psi(l_{2})=0.$
\\ We now  construct a sequence of approximation of the solution that
converges to the solution. The components $(\varphi_{n}, \
\psi_{n})$ can be elegantly determined by setting  the recursion
scheme
\begin{eqnarray}
 \left\{
\begin{array}{rcl}
\varphi_{0}&=&\beta(x-l_{1}), \ \psi_{0}=\gamma(x-l_{1}),\\
\varphi_{n}&=&\varphi_{0}+\int_{l_{1}}^{x}\int_{l_{1}}^{y}f_{1}(\varphi_{n-1}(s),\psi_{n-1}(s))dsdy,\ n\geq1,\\
\psi_{n}&=&\psi_{0}+\int_{l_{1}}^{x}\int_{l_{1}}^{y}f_{2}(\varphi_{n-1}(s),\psi_{n-1}(s))dsdy.
\ n\geq1.
\end{array}
\right.\label{60}
\end{eqnarray}
\begin{theorem}
The sequence $(\varphi_{n}, \ \psi_{n})$ defined by (\ref{60})
converges uniformly on $I$ to the unique solution $(\varphi, \
\psi)$.
\end{theorem}
\begin{proof}
We shall construct an upper bound for $\mid (\varphi_{n+1}, \
\psi_{n+1})-(\varphi_{n}, \ \psi_{n})\mid$ by induction.\\
\begin{eqnarray}
 \left\{
\begin{array}{rcl}
\mid\varphi_{1}-\varphi_{0}\mid&=&\int_{l_{1}}^{x}\int_{l_{1}}^{y}f_{1}(\varphi_{0}(s),\psi_{0}(s))dsdy\leq K_{1}\frac{(x-l_{1})^{2}}{2!},\\
\mid\psi_{1}-\psi_{0}\mid&=&\int_{l_{1}}^{x}\int_{l_{1}}^{y}f_{1}(\varphi_{0}(s),\psi_{0}(s))dsdy\leq
K_{2}\frac{(x-l_{1})^{2}}{2!},
\end{array}
\right.
\end{eqnarray}
where $K_{1}=\frac{M+MM^*}{\mid r\mid}$ and $K_{2}=\frac{\alpha M^*+\frac{M^{2}}{2}}{\mid s\mid}.$ Proceeding in the same manner,
we obtain by induction
\begin{eqnarray}
 \left\{
\begin{array}{rcl}
\mid\varphi_{n+1}-\varphi_{n}\mid &\leq& K^{n+1}_{1}\frac{(x-l_{1})^{n+2}}{(n+2)!},\\
\mid\psi_{n+1}-\psi_{n}\mid&\leq&
K^{n+1}_{2}\frac{(x-l_{1})^{n+2}}{(n+2)!}.
\end{array}
\right.
\end{eqnarray}
The  two series
$\sum_{n=0}^{\infty}K^{n+1}_{1}\frac{(x-l_{1})^{n+2}}{(n+2)!}$ and
$\sum_{n=0}^{\infty}K^{n+1}_{2}\frac{(x-l_{1})^{n+2}}{(n+2)!}$ are
absolutely  convergent series. Moreover, these series dominate the
two series $\sum_{n=0}^{\infty}[\varphi_{n+1}-\varphi_{n}]$ and
$\sum_{n=0}^{\infty}[\psi_{n+1}-\psi_{n}].$  Hence, by the
Weierstrass test, the last two infinite series
 converge absolutely and
uniformly on $I.$  If we consider the $m-th$ partial sum of these
series, we see that
$\sum_{n=0}^{m}[\varphi_{n+1}-\varphi_{n}]=\varphi_{m+1}-\varphi_{0}$
and  $\sum_{n=0}^{m}[\psi_{n+1}-\psi_{n}]=\psi_{m+1}-\psi_{0},$ that
is,  $(\varphi_{n}, \ \psi_{n})$ converges absolutely and uniformly
on $I.$ If we now define $(\varphi, \psi)=\lim_{n\rightarrow
\infty}(\varphi_{n}, \psi_{n}),$ then taking the limit as
$n\rightarrow \infty,$ we obtain
\begin{eqnarray}
 \left\{
\begin{array}{rcl}
\varphi&=&\beta(x-l_{1})+\int_{l_{1}}^{x}\int_{l_{1}}^{y}f_{1}(\varphi(s),\psi(s))dsdy,\\
\psi&=&\gamma(x-l_{1})+\int_{l_{1}}^{x}\int_{l_{1}}^{y}f_{2}(\varphi(s),\psi(s))ds
dy.
\end{array}
\right.
\end{eqnarray}
It follows that upon differentiation of this system that  $(\varphi,
\psi)$ is the solution of Eq. {(\ref{1})}. Furthermore, it is clear
that $\varphi(l_{i})=\psi(l_{i})=0, \ i=1,2.$
 \end{proof}

In view of {(\ref{60})}, the numerical solutions are then given by
\begin{eqnarray}\label{46}
 \left\{
\begin{array}{rcl}
\varphi&=&\beta(x-l_{1})+\frac{1}{r}\left[\beta \frac{(x-l_{1})^{3}}{3!}-\beta\gamma \frac{(x-l_{1})^{4}}{4!}\right]+...,\\
\psi&=&\gamma(x-l_{1})+\frac{1}{s}\left[\alpha\gamma
\frac{(x-l_{1})^{3}}{3!}-\frac{\beta^{2}}{2}
\frac{(x-l_{1})^{4}}{4!}\right]+....
\end{array}
\right.
\end{eqnarray}
If we match $(\varphi_{1},\psi_{1})$ at $x=l_{2},$ then we need to
solve
\begin{eqnarray}
 \left\{
\begin{array}{rcl}
\beta(l_{2}-l_{1})+\frac{1}{r}\left[\beta \frac{(l_{2}-l_{1})^{3}}{3!}-\beta\gamma \frac{(l_{2}-l_{1})^{4}}{4!}\right]&=&0,\\
\gamma(l_{2}-l_{1})+\frac{1}{s}\left[\alpha\gamma
\frac{(l_{2}-l_{1})^{3}}{3!}-\frac{\beta^{2}}{2}
\frac{(l_{2}-l_{1})^{4}}{4!}\right]&=&0,
\end{array}
\right.
\end{eqnarray}
we obtain   $\beta=
\sqrt{\dfrac{2(4!)^{2}rs[1+\frac{(l_{2}-l_{1})^{2}}{3!s}][1+\alpha\frac{(l_{2}-l_{1})^{2}}{3!s}]}{(l_{2}-l_{1})^{6}}}$
and $\gamma=
\dfrac{4!r[1+\frac{(l_{2}-l_{1})^{2}}{3!r}]}{(l_{2}-l_{1})^{3}}.$

In figure \ref{Fig2}, we present the variation of the numerical solutions Eqs. (\ref{46}) against of the independent variables $x$ for different values of the rescaled soliton parameter  $\alpha$.

\begin{figure}[h]
\center{
\hspace*{1.cm}\textbf{(a)} \hspace*{5.cm}\textbf{(b)}\\
 \includegraphics[width=0.4\linewidth,height=0.35\linewidth]{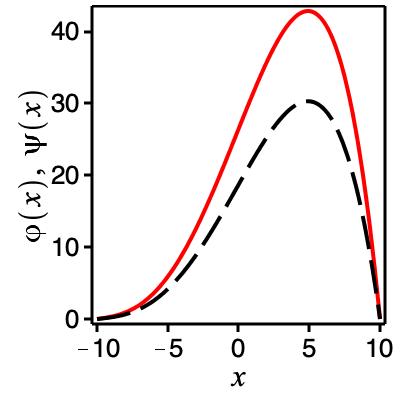}~ \includegraphics[width=0.4\linewidth,height=0.35\linewidth]{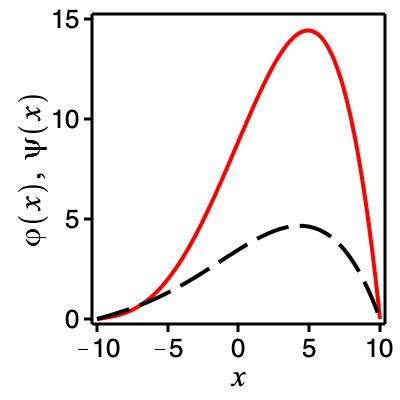}}
\caption{(Color online) The numerical solutions of Eqs. (\ref{46}) for different values of the rescaled soliton parameter  $\alpha$. (a)  for $\alpha=1$.  and (b) for $\alpha=0.1$. The other parameters are $r=1$ and $s=1.$ The solid red line represents the solution $\phi$ while the black dashed is for $\psi$ }\label{Fig2}
\end{figure}

Now if we insert the solutions of Eq. \ref{46} in the second member of Eqs. \ref{60} and performing the integrals, then using the boundary condition for $(\varphi_{2},\psi_{2})$ at $x=l_{2},$ we can obtain the second solutions $\varphi_2$ and $\psi_2$. As the mathematical expressions are more cumbersome, we plot the numerical solutions in figure. \ref{Fig3}. It is clear from these plots that when we increase the order of the recurrence, the numerical solutions converge rapidly to the exact solutions.

\begin{figure}[h]
\center{
\hspace*{1.cm}\textbf{(a)} \hspace*{5cm}\textbf{(b)}\\
 \includegraphics[width=0.4\linewidth,height=0.35\linewidth]{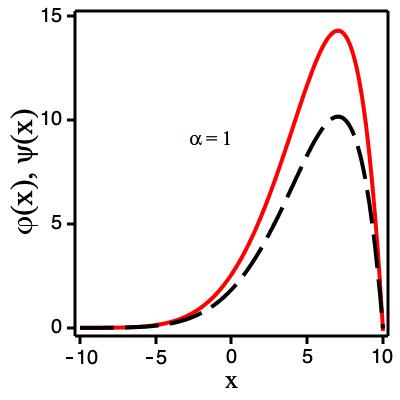}~ \includegraphics[width=0.4\linewidth,height=0.35\linewidth]{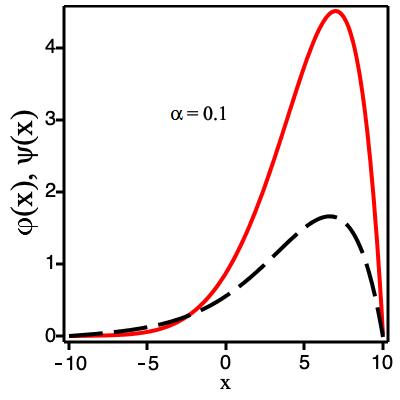}}
\caption{(Color online) The numerical solutions of Eqs. (\ref{46}) for different values of the integration constant $\alpha$. (a)  for $\alpha=1$.  and (b) for $\alpha=0.1$. The other parameters are $r=1$ and $s=1.$ The solid red line represents the solution $\varphi(x)$ while the black dashed is for $\psi(x)$ }\label{Fig3}
\end{figure}

\newpage
\section{Conclusion}\label{Sec5}
This paper is concerned with the treatment of the interaction of two-wave solitons in nonlinear quadratic media. These kinds of problems appear in different applications of nonlinear physics and play a crucial role in the stability problems of solitary waves.
The problem is presented within the framework of two coupled nonlinear differential equations, which can be solved numerically with specific boundary conditions. But the generalization to any type of boundary condition constitutes a great challenge.

With concordance to real physical problems, the boundary conditions can be chosen properly.Thus, within this framework, we have proved a theorem of existence and uniqueness for the two-wave solitons in nonlinear quadratic media.
Furthermore, we have suggested a useful technique of separation of the coupled system, and we have revealed that the formalism leads to analytic solutions.

Moreover, we have explored an interesting numerical technique, and we have used it to obtain the numerical solutions of the coupled system with suitable boundary conditions. The obtained results are in good agreement with those analytically achieved.

These crucial results open a novel class of investigations, which involve solitary waves with more coupled differential equations and more coupling terms. Some other examples of two or more-coupled solitary waves can be treated with the proposed techniques, and the results will be reported elsewhere.\\


\begin{thebibliography}{999}
\bibitem{1}A.V. Buryak and N. N. Akhmediev, Internal friction between solitons in near-integrable systems, \emph{Phys. Rev. E} 50, 31263133, 1994.
\bibitem{2} A.V. Buryak and N.N. Akhmediev, Influence of radiation on soliton dynamics in nonlinear fibre couplers,
\emph{Opt. Commun.} 110, 287-292, 1994.

\bibitem{3} A.V. Buryak and Y.S. Kivshar, Spatial optical solitons
governed by quadratic nonlinearity: erratum, \emph{Opt. Lett.,} 20,
1080-1080, 1995.

\bibitem{4} A.V Buryak, Y.S Kivshar,  Solitons due to second harmonic generation,
\emph{Phys. Lett. A,} 197, 407-412,  1995.

\bibitem{5} A.V Buryak, Y.S Kivshar, Twin-hole dark solitons, \emph{Phys. Rev. A,}
51(1), R41-R44, 1995.

\bibitem{6}A.A. Sukhorukov,  Approximate solutions and scaling
transformations for quadratic solitons., \emph{Phys. Rev. E,} 61(4),
4530-4539, 2000.

\bibitem{7}A.V. Buryak, P. D. Trapani, D. V. Skryabin, S. Trillo, Optical solitons due to quadratic nonlinearities:
from basic physics to futuristic applications, \emph{ Physics
Reports},  370, 63-235, 2002.

\bibitem{8}H. He, M. J. Werner, and P. D. Drummond, Simultaneous
solitary-wave solutions in a nonlinear parametric waveguide,\emph{
Phys. Rev. E} 54(1), 896 -911, 1996.

\bibitem{9} D.V. Skryabin and W.J Firth,
Generation and stability of optical bullets in quadratic nonlinear
media, \emph{Opt. Commun.,} 148(1-3),  79-84, 1998.

\bibitem{10} T. Peschel, U. Peschel, F. Lederer, and B. A. Malomed, Solitary waves in Bragg
gratings with a quadratic nonlinearity, \emph{Phys. Rev. E,} 55(4),
4730-4739, 1997.

\bibitem{11} Hayata, K., Koshiba, M., Multidimensional solitons in quadratic nonlinear media, \emph{ Phys. Rev. Lett.} 71, 20, 3275, 1993.

\bibitem{12}Werner, M.J., Drummond, P.D., Stongly coupled nonlinear
parametric solitary waves,\emph{Opt. Lett.,} 19, 9, 613-615, 1994.

\bibitem{13}L. Torner, E.M. Wright,  Soliton excitation and mutual locking of light beams
in bulk quadratic nonlinear crystals, \emph{J. Opt. Soc. Am. B,} 13,
5, 864-875, 1996.

\bibitem{14}A. A. Sukhorukov, Approximate solutions and scaling transformations for quadratic
solitons,\emph{Physical Review E,}61(4), 4530- 4539, 2000.

\bibitem{15}W.J. Firth,  D.V. Skryabin, Optical solitons carrying orbital angular momentum, \emph{Phys.
Rev. Lett.} 79, 2450, 1997.

\bibitem{16} N.I. Nikolov, D. Neshev, O.Bang, W.Z.
Kr\'{o}likowski, Quadratic solitons as nonlocal solitons,
\emph{Physical Review E,} 68, 036614 ~2003.


\bibitem{17} Juan Chen, Jiewen Ge, Daquan Lu, Wei Hu, A simple approach to study the boundary-induced trajectory evolution of
spatial nonlocal quadratic solitons: Basrd on the Green's function
method, \emph{ Applied Mathematics Letters,} 102, 106-108, 2020.

\bibitem{18} J.K. Hunter, B. Nachtergaele, Applied Analysis, World
Scientific, Publisching, 2001.

\end{thebibliography}
\end{document}